\documentclass[12pt]{amsart}
\usepackage{ amsmath, amsthm, amsfonts, amssymb, color}
 \usepackage{mathrsfs}
\usepackage{amsfonts, amsmath}
 \usepackage{amsmath,amstext,amsthm,amssymb,amsxtra}
 \usepackage{txfonts} 
 \usepackage[colorlinks, citecolor=blue,pagebackref,hypertexnames=false]{hyperref}
 \allowdisplaybreaks
 \usepackage{pgf,tikz}
 \usepackage{pst-plot}

\begin{document}

\hfuzz=6pt

\widowpenalty=10000

\newtheorem{cl}{Claim}
\newtheorem{theorem}{Theorem}[section]
\newtheorem{proposition}[theorem]{Proposition}
\newtheorem{coro}[theorem]{Corollary}
\newtheorem{lemma}[theorem]{Lemma}
\newtheorem{definition}[theorem]{Definition}
\newtheorem{assum}{Assumption}[section]
\newtheorem{example}[theorem]{Example}
\newtheorem{remark}[theorem]{Remark}
\renewcommand{\theequation}
{\thesection.\arabic{equation}}

\def\SL{\sqrt H}

\newcommand{\mar}[1]{{\marginpar{\sffamily{\scriptsize
        #1}}}}
\newcommand{\li}[1]{{\mar{LY:#1}}}
\newcommand{\el}[1]{{\mar{EM:#1}}}
\newcommand{\as}[1]{{\mar{AS:#1}}}

\newcommand\RR{\mathbb{R}}
\newcommand\CC{\mathbb{C}}
\newcommand\NN{\mathbb{N}}
\newcommand\ZZ{\mathbb{Z}}
\def\RN {\mathbb{R}^n}
\renewcommand\Re{\operatorname{Re}}
\renewcommand\Im{\operatorname{Im}}

\newcommand{\mc}{\mathcal}
\newcommand\D{\mathcal{D}}

\newcommand{\la}{\lambda}
\def \l {\lambda}
\newcommand{\eps}{\varepsilon}
\newcommand{\pl}{\partial}
\newcommand{\supp}{{\rm supp}{\hspace{.05cm}}}
\newcommand{\x}{\times}
\newcommand{\lag}{\langle}
\newcommand{\rag}{\rangle}

\newcommand\wrt{\,{\rm d}}

\title[Weak-type endpoint bounds for Bochner-Riesz means  ]{%
Weak-type endpoint bounds for  Bochner-Riesz means\\ for the Hermite operator}
\author{Peng Chen,     Ji Li,   Lesley A.~Ward   and  Lixin Yan }

\address{
    Peng Chen,
    Department of Mathematics,
    Sun Yat-sen University,
    Guangzhou, 510275,
    P.R.~China
    and
    School of Information Technology and Mathematical Sciences,
    University of South Australia,
    Mawson Lakes SA 5095,
    Australia}
\email{
    chenpeng3@mail.sysu.edu.cn}
\address{
    Ji Li,
    Department of Mathematics,
    Macquarie University, NSW 2109,
    Australia
    }
\email{
    ji.li@mq.edu.au}
\address{
    Lesley A.~Ward,
    School of Information Technology and Mathematical Sciences,
    University of South Australia,
    Mawson Lakes SA 5095,
    Australia}
\email{
    lesley.ward@unisa.edu.au}
\address{
    Lixin Yan,
    Department of Mathematics,
    Sun Yat-sen University,
    Guangzhou, 510275,
    P.R.~China}
\email{
    mcsylx@mail.sysu.edu.cn}

\subjclass[2010]{Primary: 42B15; Secondary: 42B08, 42C10.}

 \date{\today}
 \keywords{Bochner-Riesz means, Hermite operator,  weak-type bounds, $L^p$ eigenfunction bounds,  finite speed of
 propagation property.}

\begin{abstract}
We obtain weak-type $(p, p)$ endpoint bounds for Bochner-Riesz
means for the Hermite operator $H = -\Delta + |x|^2$ in
${\mathbb R}^n, n\geq 2$ and for other related operators, for
$1\leq p\leq 2n/(n+2)$, extending earlier results of Thangavelu
and of Karadzhov.
\end{abstract}

\maketitle


\section{Introduction}
\setcounter{equation}{0}

Convergence of the Bochner-Riesz means    on
Lebesgue $L^p$ spaces is one of the classical problems in harmonic analysis.
Let us begin with recalling the Bochner-Riesz means operator $S_R^\delta$ on
$\RN$ which  is  defined by,  for $\delta\ge 0$ and $R>0$,
\begin{eqnarray}\label{e1.1}
\widehat{S^{\delta}_Rf}(\xi)
=\left(1-{|\xi|^2\over R^2}\right)_+^{\delta} \widehat{f}(\xi),  \quad {\rm for\, all}\,\,{\xi \in \RN}.
\end{eqnarray}
 Here $\widehat{f}\,$ denotes the Fourier  transform of $f$ and $(x)_+:=\max\{0,x\}$ for $x\in \mathbb R$.
A natural problem is to characterize  the  optimal range of $\delta$ for which $S^\delta_R$  is  bounded  on $L^p(\RN)$. The  Bochner-Riesz    conjecture is that, for $n\geq 2$ and $1\le p \le 2n/(n+1)$, $S_{R}^\delta $ is bounded on $L^p(\mathbb R^n)$ if and only if
\begin{equation}
\label{e1.2}
\delta> \delta(p)=  n \left({1\over p}-{1\over 2} \right)-{1\over 2}.
\end{equation}
It was shown by Herz that for a given $p$  the above condition on  $\delta$  is
necessary; see \cite{He}.  Carleson and Sj\"olin \cite{CS} proved the conjecture when $n=2$.  Afterward substantial progress has been
made \cite{F1, tvv, Lee,  BoGu}, but the conjecture still remains open for $n\ge 3$ and $p$ close to $2n/(n+1).$
We refer the reader to     Stein's monograph  \cite[Chapter IX]{St2}  and Tao \cite{Ta3} for historical background and more on
the Bochner-Riesz conjecture.

Concerning the endpoint estimates (for $\delta=\delta(p)$) of the Bochner-Riesz means, it is natural to conjecture  that $ S_R^{\delta(p)}  $
is of weak-type $(p, p)$ for $1\le p < 2n/(n+1)$. In the special case $n=2$ the weak-type endpoint conjecture was proved by Seeger \cite{Se} for the full range of $p\in [1,4/3)$.
  In higher dimensions the weak-type endpoint estimate was proved  by Christ \cite{C1, C2}  for the range $1\leq p < 2(n+1)/(n+3)$, making use of the well-known
  $(p, 2)$ restriction theorem
  of Stein-Tomas~\cite[p. 386]{St2}. The weak-type endpoint estimate
 for $p = 2(n+1)/(n+3)$ was  proved by Tao~\cite{T2}. As  shown  by Tao \cite{Ta1},
 the weak-type endpoint Bochner-Riesz conjecture is equivalent to the standard Bochner-Riesz conjecture.

 Inspired by the works of Christ and Tao \cite{C1, C2, T2},
  Ouhabaz, Sikora and the first and fourth authors of this paper  extended  the above results   to
the Bochner-Riesz means associated to  second-order elliptic differential operators
$L$ on $\RN$ which are self-adjoint and formally non-negative; see~\cite{COSY}.
Such an operator $L$ admits a  spectral resolution
$$
Lf=\int_0^{\infty}\lambda dE_L(\lambda) f, \ \ \ \ f\in L^2(\RN),
$$
where  $E_L(\lambda)$ is the projection-valued measure supported on the spectrum of $L$.
Notice that the spectrum  of $L$ may be continuous, discrete, or a combination of both.
By the spectral theorem,
the Bochner-Riesz means for $L$ of order $\delta\geq 0$ with $R>0$ are defined by
 \begin{equation}\label{e1.3}
 S^{\delta}_R(L) f =   \int_0^{R^2} \left(1-\frac{\lambda}{R^2}\right)^{\delta}_+dE_L(\lambda) f
   \end{equation}
 for $f\in L^2(\RN)$. In the special case when  $-L$ is the standard Laplace operator
$\Delta=\sum_{i=1}^n\partial_{x_i}^2$ on $\RN$,  $S^{\delta}_R(-\Delta)$
  coincides with the usual  $S^{\delta}_R$.
   It was proved in \cite[Theorem I.24]{COSY}
  that   if $L$  satisfies the finite speed of propagation property   \eqref{FS} (see Section 2 below),
 then for some $p$ with $1\leq p <2$, the spectral measure estimate
   \begin{equation*}\label{rp}
 \tag{${\rm R_p}$}
\|dE_{\sqrt L}(\lambda)\|_{L^p(\RN)\to  L^{p'}(\RN)  }\le C\ \lambda^{{n}(\frac{1}{p}-\frac{1}{p'})-1}, \ \ \lambda>0\
 \end{equation*}
implies  weak-type $(p,p)$ estimates for  $S_R^{\delta(p)} (L)$,
uniformly in $R$.  This  recovers  the  known results in \cite{C1, C2, T2}.
To understand  the condition \eqref{rp}, we recall that for $\lambda>0$,
the restriction operator $R_\lambda$ is  given by    $R_\lambda(f)(\omega) =\hat{f}({\lambda} \omega),$
where $\omega\in   {\bf S}^{n-1}$ (the unit sphere).
Then
 $
d E_{\sqrt{-\Delta}}(\lambda) =(2\pi)^{-n} \lambda^{n-1}R_\lambda^*R_\lambda,
 $
and  for $p$ with $1\le p\le 2(n+1)/(n+3)$ the Stein-Tomas $(p,2)$ restriction theorem \cite[p.386]{St2}
is equivalent to the estimate \eqref{rp}.
The condition \eqref{rp}  is  valid for a broad class of second-order
elliptic operators  such as
scattering operators  on ${\mathbb R}^3$  and  Schr\"odinger operators $-\Delta + V$ on $\RR^n$,
where $V$ is smooth and decays sufficiently fast at infinity. See \cite[Propositions III.3 and III.6]{COSY}.

The  condition   \eqref{rp} implies that
  the point spectrum of $L$ is empty.
In particular,
\eqref{rp}  does  not hold  for elliptic
operators on  compact manifolds, nor for the Hermite operator $H= -\Delta+|x|^2$ on~$\RN$.
In the case of the Laplace-Beltrami operator $\Delta_g$ on a compact smooth Riemannian manifold $(M, g)$ of dimension $n\geq 2,$
  Sogge \cite{Sog2} used a Fourier transform side argument to prove that under an additional curvature assumption, one has a (discrete) $(p, 2)$ restriction
  theorem for all $p$ with $1\leq p\leq 2(n+1)/(n+3)$, namely
 \begin{equation*}\label{sp}
 \tag{${\rm S_p}$}
\big\|E_{\sqrt{\Delta_g}}[\lambda, \lambda+1)\big\|_{L^p(M)\to L^{2}(M)} \leq
C (1+\lambda)^{\delta(p)}, \ \ \ \ \lambda \geq 0.
\end{equation*}
In Christ and Sogge \cite{ChS}, it was shown using the condition~\eqref{sp} that $S_R^{\delta(1)}(\Delta_g)$ is weak  $(1, 1)$ uniformly in $R$.
Later, weak-type $(p,p)$ estimates for  $ S_R^{\delta(p)}(\Delta_g) $
were proved by Seeger \cite{Se1} when $1<p < 2(n+1)/(n+3)$ and by Tao \cite{T2} when $p=2(n+1)/(n+3)$. See also \cite[Proposition III.2]{COSY}.

\smallskip

 The purpose of this paper  can be viewed as a continuation of the above body of work  on the weak-type   $L^p$   mapping properties of the Bochner-Riesz summation
for
   the Hermite operator
 $H=-\Delta + |x|^2
$  on $ \RR^n, n\geq 2$, and for other related operators. For the Hermite operator, it is known that
the spectral decomposition of $H$ is given by the Hermite expansion; see \cite{Th1}. Let
$h_{\alpha}(x), \alpha\in {\mathbb N}^n$, be the normalized Hermite functions which are
  eigenfunctions for $H$ with eigenvalues $(2|\alpha|+n)$ where $|\alpha|=\alpha_1 +\cdots +\alpha_n$.
Thus   every $f\in L^2(\RN)$  has the Hermite expansion
\begin{eqnarray} \label{e1.4}
f=\sum_{\alpha}(f, h_{\alpha})h_\alpha,
\end{eqnarray}
where the sum is extended over all multi-indices $\alpha\in {\mathbb N}^n.$
Then the  Bochner-Riesz means for $H$ of order $\delta\geq 0$ with $R>0$ as defined in equation~\eqref{e1.3} with $L=H$ coincide with
\begin{eqnarray}\label{e1.5}
S_R^{\delta}(H)f
 =
\sum_{k=0}^{\infty} \left(1-{2k+n\over R^2}\right)_+^{\delta} P_k f,
\end{eqnarray}
where $P_k$ are the projections
\begin{eqnarray} \label{e1.6}
P_kf=\sum_{|\alpha|=k}(f, h_{\alpha})h_\alpha.
\end{eqnarray}
The Hermite expansion \eqref{e1.4} and the corresponding Bochner-Riesz means \eqref{e1.5} were studied
in~\cite{Th1}. When $n\geq 2$, the conjecture  is that   the operators
$S_R^{\delta}(H)$ are bounded on $L^p(\RN)$, uniformly in $R>0$,  if and only if $\delta>\delta(p)$,
where $\delta(p)$ is the same critical index as defined in~\eqref{e1.2} for the Bochner-Riesz means in the case of the standard Laplacian on $\RN$
(see~\cite[p.259]{Th4}).
In \cite{Th1}, Thangavelu proved that the conjecture is true when $p=1$ and that for a given $p$ the above condition on
$\delta$ is necessary. In 1994, Karadzhov  \cite{Kar} proved the conjecture in the range  $1\leq p\leq 2n/(n+2).$
The main ingredient in the proof of these results
is   to establish the  following restriction type theorem
  \begin{eqnarray}\label{e1.7}
\|P_k f\|_{L^2(\RN)}\leq Ck^{(\delta(p)-1/2)/2}\|f\|_{L^p(\RN)},  \ \ \  {\rm for\, all}\, k\in{\mathbb N}
\end{eqnarray}
for the spectral projection operators $P_k$ for $1\leq p\leq 2n/(n+2),$
which is an adaptation of the arguments from  \cite{F3, Sog1} that
the restriction theorem  implies    Bochner-Riesz summation theorems for $L^p(\RN)$.
For more on the Bochner-Riesz summation for the Hermite operator, see also \cite{Th2, Th3, Th4}.

The main goal of this paper is to extend the results of \cite{Kar, Th1} to  weak-type  endpoint results for the range  $1\leq p\leq 2n/(n+2).$
We first recall that for $1\leq p<\infty$, a function $f$ is said to be in weak $L^p({\RN})$, written $f\in L^{p,\infty}(\RN)$, if
$$
\|f\|_{L^{p, \infty}(\RN)}:=\sup_{\alpha>0}  \alpha \left|\{x: |f(x)|>\alpha\}\right|^{1/p}   <\infty.
$$
We can now state our main result, which we put in context in Figure~\ref{Figure1}(b).

\vspace{-0.2cm}
 \begin{theorem}\label{th1.1} \, For $n\geq 2$ and  $1\leq p\leq 2n/(n+2),$
the Bochner-Riesz means $S_R^{\delta(p)} (H)$ are of weak-type $(p,p)$ uniformly in $R$. That is, there exists a constant $C>0$ independent
of $R$ such that
 \begin{eqnarray*}
   \|S_R^{\delta(p)}(H) f\|_{L^{p, \infty}(\RN)} \leq C\|f\|_{L^p(\RN)}, \quad\quad\text{for all $f\in L^p(\RN)$ and all $R>0$}.
 \end{eqnarray*}
\end{theorem}

\vspace{-0.5cm}
\begin{figure}[htb]
  \begin{minipage}[t]{8.0cm}
\centering
\hspace{-1cm}
\begin{pspicture}(5.0,4.0)
    \psset{xunit=5cm, yunit=1.39cm}
    \newgray{gray0}{.95}
    \newgray{gray1}{.7}
    \newgray{gray2}{.75}
    \newgray{gray3}{.55}
    \pspolygon[linestyle=none,fillstyle=solid,fillcolor=gray3](.5,0)(.2,.7)(0,1.5)(0,2.5)(1,2.5)(1,1.5)(.8,.7)
    \pspolygon[linestyle=none,fillstyle=solid,fillcolor=gray2](.5,0)(.8,.7)(.625,0)
\pspolygon[linestyle=none,fillstyle=solid,fillcolor=gray2](.5,0)(.2,.7)(.375,0)
    \pspolygon[linestyle=none,fillstyle=solid,fillcolor=gray2](.5,0)(.8,.7)(.625,0)
    \pspolygon[linestyle=none,fillstyle=solid,fillcolor=gray3](.5,0)(0,2)(0,2.8)(1,2.8)(1,2)
    \psline[linewidth=.5pt]{->}(0,0)(1.2,0)
    \uput[d](1.2,0){\large$\frac{1}{p}$}
    \psline[linewidth=.5pt]{->}(0,0)(0,3)
    \uput[l](0,2.9){ $\delta$}
    \psline[linestyle=dashed](.375,0)(0,1.5)
\psline[linestyle=dashed](.5,0)(.8,.7)
\psline[linestyle=dashed](.5,0)(.2,.7)
    \psline[linewidth=1pt](.8,.7)(1,1.5)
    \psline[linestyle=dashed](.625,0)(.8,.7)
    \psline[linestyle=dotted](.8,0)(.8,.7)
    \psline[linestyle=dotted](.2,0)(.2,.7)
    \psline[linestyle=dashed](1,0)(1,1.5)
    \psdot(1,1.5)
    \psdot(.8,.7)
    \uput[l](1.12,1.28){\large A}
    \uput[l](.92,.55){\large B}
    \uput[l](0,1.5){\large$\frac{n-1}{2}$}
    \uput[d](.5,0){\large$\frac{1}{2}$}
      \uput[d](0,0){$0$}
      \uput[d](1,0){$1$}
    \uput[d](.625,0){\large$\frac{n+1}{2n}$}
    \uput[d](.375,0){\large$\frac{n-1}{2n}$}
    \uput[d](.8,0){\large$\frac{n+3}{2(n+1)}$}
    \uput[d](.2,0){\large$\frac{n-1}{2(n+1)}$}
\rput(0.5,-0.9){{\footnotesize Figure 1(a). Results for Laplacian operator }}
\rput(0.5,-1.2){{\footnotesize on $\RN$ or on compact manifolds.}}

\end{pspicture}
\vspace{1cm}
  \end{minipage}
\hfill
\begin{minipage}[t]{8.0cm}
\centering
\begin{pspicture}(5.0,4.0)
    \psset{xunit=5cm, yunit=1.39cm}
    \newgray{gray0}{.95}
    \newgray{gray1}{.7}
    \newgray{gray2}{.75}
    \newgray{gray3}{.55}
    \pspolygon[linestyle=none,fillstyle=solid,fillcolor=gray3](.5,0)(.15,.895)(0,1.5)(0,2.5)(1,2.5)(1,1.5)(.85,.895)
    \pspolygon[linestyle=none,fillstyle=solid,fillcolor=gray2](.5,0)(.85,.895)(.625,0)
\pspolygon[linestyle=none,fillstyle=solid,fillcolor=gray2](.5,0)(.15,.895)(.375,0)
    \pspolygon[linestyle=none,fillstyle=solid,fillcolor=gray2](.5,0)(.85,.895)(.625,0)
    \pspolygon[linestyle=none,fillstyle=solid,fillcolor=gray3](.5,0)(0,2)(0,2.8)(1,2.8)(1,2)
    \psline[linewidth=.5pt]{->}(0,0)(1.2,0)
    \uput[d](1.2,0){\large$\frac{1}{p}$}
    \psline[linewidth=.5pt]{->}(0,0)(0,3)
    \uput[l](0,2.9){ $\delta$}
    \psline[linestyle=dashed](.375,0)(0,1.5)
\psline[linestyle=dashed](.5,0)(.85,.895)
\psline[linestyle=dashed](.5,0)(.15,.895)
    \psline[linewidth=1pt](.85,.895)(1,1.5)
    \psline[linestyle=dashed](.625,0)(.85,.895)
    \psline[linestyle=dotted](.85,0)(.85,.895)
    \psline[linestyle=dotted](.15,0)(.15,.895)
    \psline[linestyle=dashed](1,0)(1,1.5)
    \uput[l](1.12,1.28){\large A}
    \uput[l](1.0,.75){\large $\rm B'$}
    \psdot(1,1.5)
    \psdot(.85,.895)
    \uput[l](0,1.5){\large$\frac{n-1}{2}$}
    \uput[d](.5,0){\large$\frac{1}{2}$}
      \uput[d](0,0){$0$}
      \uput[d](1,0){$1$}
    \uput[d](.625,0){\large$\frac{n+1}{2n}$}
    \uput[d](.375,0){\large$\frac{n-1}{2n}$}
    \uput[d](.85,0){\large$\frac{n+2}{2n}$}
    \uput[d](.15,0){\large$\frac{n-2}{2n}$}
\rput(0.5,-0.9){{\footnotesize Figure 1(b). Results for Hermite operator, }}
\rput(0.5,-1.2){{\footnotesize including our result.}}
\end{pspicture}
  \end{minipage}
  \vspace{0.6cm}
 \caption{ \footnotesize{ Schematic diagrams, in dimension $n\geq 3$, summarizing known results on boundedness
 of Bochner-Riesz means $S_R^\delta$, in Figure 1(a) for the Laplacian operator and in Figure 1(b) for the Hermite operator. For
 each point $(1/p,\delta)$ in the dark gray regions, $S_R^\delta$ is bounded on $L^p$; for $(1/p,\delta)$
 in the light gray triangles, the boundedness of $S_R^\delta$ is unknown or partial results are known; and for
 $(1/p,\delta)$ in the white triangles, $S_R^\delta$ is not bounded on $L^p$.  For $(1/p,\delta)$
  on the line segment $AB$, where $\delta=\delta(p)$, in Figure 1(a), the (Laplacian) Bochner-Riesz means $S_R^\delta=S_R^{\delta(p)}$ satisfy the
  weak-type endpoint estimate. Our result in Theorem~\ref{th1.1} is that for $(1/p,\delta)$
  on the line segment $AB'$ in Figure 1(b), the (Hermite) Bochner-Riesz means $S_R^\delta(H)=S_R^{\delta(p)}(H)$ satisfy the   weak-type endpoint estimate. Note that $A$ represents the same point in both figures, but $B\neq B'$. }
  }\label{Figure1}
 \end{figure}
\vspace{1cm}
As a consequence of this theorem, we have that when $f\in L^p(\RN)$,
 the operator $S_R^{\delta(p)}(H) f$  converges in measure to $f$.
By  this we mean that for each $\alpha>0$,
\begin{eqnarray*}
 \big|\{x: |S_R^{\delta(p)} (H)f(x)- f(x)|>\alpha \}\big|  \rightarrow  0 \ {\rm as}\ R \rightarrow \infty.
   \end{eqnarray*}
   In fact $S_R^{\delta(p)}(H) f\rightarrow f$ in the $L^{p, \infty}$ quasi-norm, that is,
\begin{eqnarray*}
 \|S_R^{\delta(p)}(H) f - f\|_{L^{p, \infty}(\RN)} \rightarrow 0.
   \end{eqnarray*}
  This result is of course considerably weaker than almost-everywhere convergence, and, in fact,
  at the critical index $\delta(p)$ one does not generally have almost-everywhere convergence of the Riesz means
  to a given $L^1$ function; see Stein and Weiss~\cite{SW}.

We would like to mention that our restriction-type
condition~\eqref{e1.7} is weaker than the classical
restriction-type condition~\eqref{rp}. To compensate for this
difference, when proving the weak-type $L^p$ estimates for
$S_R^{\delta(p)} (H)$ in our Theorem~\ref{th1.1}, we need an
\emph{a priori} estimate
\begin{eqnarray}\label{e1.8}
    \|(1+H)^{-(\delta(p)+1/2)/2}\|_{L^2(\RN)\to L^{p,\infty}(\RN)}
    \leq C
\end{eqnarray}
for $1\leq p\leq 2n/(n+2)$, and this is a crucial observation
in our paper. Then Theorem~\ref{th1.1} is proved by using the
\emph{a priori} estimate~\eqref{e1.8}, along with the $L^p$
eigenfunction bounds \eqref{e1.7} for the Hermite operator, and
the approach in the work of Christ~\cite{C1, C2} and Tao~\cite{T2}.
Their approach is based on $L^2$ Calder\'on-Zygmund techniques
(as used in Fefferman \cite{F1}), a spatial decomposition of
the Bochner-Riesz summation, and the fact that if the inverse
Fourier transform $F$ is supported on a set of width $R$, then
by the finite speed of propagation property the operator
$F(\sqrt L)$ is supported in a $CR$-neighbourhood of the
diagonal.

We outline our proof of Theorem~\ref{th1.1} here,
highlighting the point where it differs from the approach of Christ~\cite{C1, C2} and Tao~\cite{T2}.
We first use $L^2$ Calder\'on-Zygmund techniques to
decompose the function~$f$ into $f = g + \sum_j b_j$.
Next we make a decomposition (Lemmas~\ref{le2.3}
and~\ref{le2.4}) of the Bochner-Riesz multiplier function,
corresponding to this Calder\'on-Zygmund decomposition, in such
a way that the main contribution acting on $b_j$ is a
multiplier operator $n_j(\sqrt H)$,
where the support of~$n_j : \mathbb{R} \to \mathbb{R}$ is mostly concentrated on a set whose radius goes
like the reciprocal of the radius of the support of $b_j$.
For the ``good" part $g$ and for those
$b_j$ which have small support, the argument is similar to that
in~\cite{T2}. However for those
$b_j$ with large support, following the argument in~\cite{T2}, we get an extra factor
in the upper bound for the $L^2$ estimate of $n_j(\sqrt H)b_j$ (see estimate~\eqref{e3.11} below), compared to the situation treated in Christ~\cite{C1, C2} and Tao~\cite{T2},
where the
operator~$L$ satisfies the restriction type estimate~\eqref{rp} or the manifold on which $f$ is defined is compact.
We overcome the obstacle posed by this extra factor by applying our \emph{a priori}
estimate~\eqref{e1.8} and a modification of the argument in~\cite{T2}. See Section~3 for details, specifically where we use
the \emph{a priori} estimate~\eqref{e1.8} to deduce
estimate~\eqref{e3.10} from estimate~\eqref{e3.11}.

The paper is organized as follows.
In Section 2 we provide some preliminary results, which we need later,
mainly to prove \eqref{e1.8} and a few technical lemmas.
The proof of Theorem~\ref{th1.1} is given in Section 3. In Section 4 we discuss some extensions
of Theorem~\ref{th1.1}
for other operators related to the Hermite operator $H$.

\medskip

\section{Preliminaries}
\setcounter{equation}{0}
For brevity, in the rest of the whole paper, for $1\le p\le+\infty$, we write $L^p$ for $L^p(\RN)$, $L^{p,\infty}$ for $L^{p,\infty}(\RN)$, and so on. We denote the
norm of a function $f\in L^p$ by $\|f\|_p$ and if $T$ is a bounded linear operator from $
L^p$ to $L^q$, $1\le p, \, q\le+\infty$, we write $\|T\|_{p\to q} $ for
the  operator norm of $T$.

In this section, we mainly consider Schr\"odinger operators~$H_V$ similar to the Hermite operator $H$, that is,
  $ H_V= - \Delta + V$ on $ \RR^n$ for $n\ge 2$,
with a  positive potential $V$ which satisfies the conditions
\begin{equation}\label{e2.1}
V \sim |x|^2, \quad |\nabla V| \sim |x|, \quad |\partial_x^2 V| \le 1.
\end{equation}
We first state some basic properties of $H_V$ in Lemmas~\ref{le2.1} and~\ref{le2.5}. Then we prove
estimate~\eqref{e1.8} for $H_V$ in Lemma~\ref{le2.2}. Finally, we state two technical lemmas, Lemmas~\ref{le2.3} and~\ref{le2.4}, for
decompositions of Bochner-Riesz multiplier functions.

The operator   $H_V$ is a self-adjoint
operator on $L^2$.  Since the potential $V$ is nonnegative,
the semigroup kernels ${\mathcal K}_t(x,y)$  of the operators $e^{-tH_V}$  satisfy
\begin{eqnarray}
0\leq {\mathcal K}_t(x,y)\leq h_t(x-y)
\label{e2.2}
\end{eqnarray}
for all $x$, $y\in\RN$ and $t>0$, where
\begin{equation}\label{e2.3}
h_t(x-y)=\frac{1}{(4\pi t)^{n/2}}\,\exp\bigg(-\frac{|x-y|^2}{4t}\bigg)
\end{equation}
is the kernel of the classical heat semigroup
$\{T_t\}_{t>0}=\{e^{t\Delta}\}_{t>0}$ on $\RN$.

To formulate the finite speed of propagation  property for the wave equation corresponding to an operator $H_V$, we
set
\begin{equation*}
\D_r:=\{ (x,\, y)\in \RN \times \RN:  |x-y|\le r \}.
\end{equation*}
Given an  operator $T$ from $L^p$ to $L^q$,
we write
\begin{equation}\label{eq2.4}
\supp  K_{T} \subseteq \D_r
\end{equation}
 if $\langle T f_1, f_2 \rangle = 0$ whenever $f_1 \in L^p(B(x_1,r_1))$ and $f_2 \in L^{q'}( B(x_2,r_2))$ with $r_1+r_2+r < |x_1-x_2|$, where as usual $1/q+1/q'=1$.
 Note that  if $T$ is an
integral operator with   {kernel $K_T$}, then (\ref{eq2.4}) coincides
with the  standard meaning of $\supp  K_{T}  \subseteq \D_r$,
 namely, $K_T(x, \, y)=0$ for all $(x, \, y) \notin \D_r$.

 Following \cite{CGT}, given  a nonnegative self-adjoint operator $L$
on $L^2$ we say that $L$ satisfies the \emph{finite speed of
propagation property} if
 \begin{equation*}\label{FS}
  \tag{FS}
\supp  K_{\cos(t\sqrt{L}\,\,)} \subseteq \D_t, \quad {\rm for\,all\,} t> 0\,.
\end{equation*}
From \eqref{e2.2} and \eqref{e2.3}, it follows (see for example \cite{CouS})  that the operator $H_V$ satisfies the finite speed of
propagation property \eqref{FS}. Then we have
  the following result.

\begin{lemma}\label{le2.1}
Assume  that $F$ is an even bounded Borel function with Fourier
transform  $\widehat{F} \in L^1(\RR)$ and that
$\supp \widehat{F} \subseteq [-r, r]$.
Then the kernel $K_{F(\sqrt{H_V})}$ of the operator $F(\sqrt{H_V})$ satisfies
$$
\supp K_{F(\sqrt{H_V})} \subseteq \D_r.
$$
\end{lemma}

\begin{proof}
If $F$ is an even function, then by the Fourier inversion formula,
$$
F(\sqrt{H_V}) =\frac{1}{2\pi}\int_{-\infty}^{+\infty}
  \widehat{F}(t) \cos(t\sqrt{H_V}) \;dt.
$$
But $\supp\widehat{F} \subseteq [-r,r]$,
and the lemma then follows  from \eqref{FS}.
\end{proof}

\smallskip
We also have the following  result.

 \begin{lemma}\label{le2.5}  Let
 $\{Q_k\}_{k\in{\mathbb N}}$ be  a family of continuous real-valued functions such that
  $\sum_k |Q_k(\lambda)|^2\leq A$ for some constant
 $A$ independent of $\lambda\in \mathbb{R}$. Then for every sequence of functions $\{f_k\}_{k\in{\mathbb N}}$ on $\RN,$
\begin{eqnarray}\label{eq4.22}
\big\|\sum_{k\in{\mathbb N}}Q_k(\sqrt{H_V})f_k\big\|_2^2\leq   A\sum_{k\in{\mathbb N}}\big\|f_k\big\|_2^2.
\end{eqnarray}
 \end{lemma}

\begin{proof}
 The proof of Lemma~\ref{le2.5} is given in \cite[Lemma I.28]{COSY}.
\end{proof}

Let $H_V= - \Delta + V$ with a  positive potential $V$ satisfying \eqref{e2.1}.
It is shown in \cite[Corollary 6.3]{CLSY}
 that for each $\nu>0,$
\begin{eqnarray}\label{e2.44}
 \|(1+H_V)^{-\gamma/2}\|_{2\to p}\leq C, \ \ \ {\rm where\,}\gamma=n(1/p-1/2) +\nu
 \end{eqnarray}
for   $1\leq p\leq 2n/(n+2)$ (see  also \cite[Lemma 7.9]{DOS} for $p=1$). To prove Theorem~\ref{th1.1},
we need the following endpoint version of~\eqref{e2.44}.

  \begin{lemma}\label{le2.2} Let $H_V= - \Delta + V$
with a  positive potential $V$ satisfying \eqref{e2.1}.
Then
\begin{eqnarray}\label{e2.4}
 \|(1+H_V)^{-\gamma/2}\|_{L^2\to L^{p,\infty}}\leq C, \ \ \ {\rm where\,} \gamma=n(1/p-1/2),
 \end{eqnarray}
for   $1\leq p\leq 2n/(n+2)$.
 \end{lemma}

\begin{proof}
To prove \eqref{e2.4}, we put $M_g(f):=fg$ and $M:=M_{\sqrt{1+V}}$. We observe  that
 $$
 \|(1+H_V)^{1/2}f\|_2^2=\langle(1+H_V)f,f\rangle \geq \langle M^2f,f\rangle=\|Mf\|_2^2.
 $$
 Now by the L\"owner-Heinz inequality for arbitrary quadratic forms $B_1$ and $B_2$, if $B_1\geq B_2\geq 0$, then $B_1^\alpha\geq B_2^\alpha$ for
 $0\leq \alpha\leq 1$. Hence $$\langle(1+H_V)^\alpha f,f\rangle \geq \langle M^{2\alpha}f,f\rangle.$$ Thus, for $\alpha\in [0,1]$,
 \begin{eqnarray}\label{e2.5}
 \|M^\alpha(1+H_V)^{-\alpha/2}\|_{2\to 2}\leq C .
 \end{eqnarray}
 For $\alpha=1$ the operator $M^\alpha(1+H_V)^{-\alpha/2}$ is of first-order Riesz transform type, and a standard argument yields,
 \begin{eqnarray*}
 \|M(1+H_V)^{-1/2}\|_{L^1\to L^{1,\infty}}\leq C;
\end{eqnarray*}
 see \cite[Theorem 11]{S2}. Then by an interpolation theorem for Lorentz spaces~\cite[Theorem 1.4.19]{Gra}, which as noted there can be seen as the off-diagonal extension of Marcinkiewicz's interpolation theorem, we have for each $q\in (1,2)$,
 \begin{eqnarray}\label{e2.6}
 \|M(1+H_V)^{-1/2}\|_{L^{q,\infty}\to L^{q,\infty}}\leq C.
\end{eqnarray}
By H\"older's inequality for weak spaces (see  for example \cite[Exercise 1.1.15]{Gra}),
 for all $q_1\geq q_2\geq 1$ with $s=(1/q_2-1/q_1)^{-1}$,
\begin{eqnarray}\label{e2.7}
 \|M^{-\alpha}\|_{L^{q_1,\infty}\to L^{q_2,\infty}}\leq C \left\|\big(\sqrt{1+V}\big)^{-\alpha}\right\|_{L^{s,\infty}}.
 \end{eqnarray}
Recall that  $\gamma=n(1/p-1/2)$.  Write
 \begin{eqnarray}\label{e2.8}
(1+H_V)^{-\gamma/2}=\big(M^{-1}M(1+H_V)^{-1/2}\big)^{[\gamma]} M^{[\gamma]-\gamma} M^{\gamma-[\gamma]}(1+H_V)^{([\gamma]-\gamma)/2}.
\end{eqnarray}
Since $V(x)\sim |x|^2$, choosing
$s=n/\alpha$ in \eqref{e2.7} gives
$$
\left\|\big(\sqrt{1+V}\big)^{-\alpha}\right\|^s_{L^{s,\infty}}
\leq C\sup_{\lambda>0}\lambda^s\{x\in\mathbb{R}^n: (\sqrt{1+V})^{-\alpha}>\lambda\}\leq C\lambda^s \lambda^{-n/\alpha}\leq C.
$$
Define $p_0$ by
 ${1/p_0}=
{(\gamma-[\gamma])/n} +{1/2}
 $,
and for each $1\leq i\leq [\gamma]-1$  define $p_i$ by setting $1/p_{i+1}- 1/p_i=1/n$, so $p_{[\gamma]}=p$.
Now multiple composition of the operators from  \eqref{e2.5}, \eqref{e2.6}  and \eqref{e2.7}, in combination with \eqref{e2.8},  yields
\begin{eqnarray*}
 &&\hspace{-1cm}\|(1+H_V)^{-\gamma/2}\|_{L^2 \to L^{p,\infty}}\\
 &\leq& \|M^{\gamma-[\gamma]}(1+H_V)^{([\gamma]-\gamma)/2}\|_{L^2 \to L^{2,\infty} }
 \|M^{[\gamma]-\gamma}\|_{L^{2,\infty} \to L^{p_0,\infty} }
\prod_{i=0}^{[\gamma]-1}\|M^{-1}M(1+H_V)^{-1/2}\|_{L^{p_i,\infty} \to L^{p_{i+1},\infty} }\\
 &\leq&  C.
\end{eqnarray*}
This completes the proof of \eqref{e2.4}.
 \end{proof}

The proof  of Theorem~\ref{th1.1} also requires the following two   technical lemmas for
decompositions of Bochner-Riesz multiplier functions.

\begin{lemma}\label{le2.3}
 For each integer $k\leq 0$ there exists a decomposition of the Bochner-Riesz multiplier function $S_R^{\delta(p)}(\lambda^2)$ as follows:
\begin{eqnarray}\label{e2.9}
S_R^{\delta(p)}(\l^2):=\Big(1-\frac{\l^2}{R^2}\Big)^{\delta(p)}_+=\eta_k(\l)n_k(\l)+S_R^{\delta(p)}(\l^2)n_k(\l),\quad\quad \lambda\in \mathbb{R},
\end{eqnarray}
such that
\begin{itemize}
\item[(a)]
 The functions  $n_k$ are  even and their Fourier transforms are supported in
$[-2^k/R,\, 2^k/R]$, that is,  $\supp \widehat{n_k}\subset
[-2^k/R,\, 2^k/R]$;

\item[(b)] The functions $\eta_k$ are continuous and even, and
$\sum_{k=-\infty}^0|\eta_k(\l)|^2\leq C$ with $C$ independent of
$\l$ and $R$;

\item[(c)]  For arbitrary large  $N \in \NN$ there exists a constant $C$ such that
$$|n_k(\l)|\leq
C\Big(1+{2^k|\l|\over R}\Big)^{-N}.$$
\end{itemize}
\end{lemma}

\begin{proof}
For the proof, we refer the reader to \cite[Lemma I.26]{COSY}. See also \cite[Lemma 2.1]{T2}.
\end{proof}

\begin{lemma} \label{le2.4}
For each integer $k>0$, there exists a decomposition of the Bochner-Riesz multiplier function $S_R^{\delta(p)}(\lambda^2)$ as follows:
\begin{eqnarray}\label{e2.99999}
S_R^{\delta(p)}(\l^2)=\Big(1-\frac{\l^2}{R^2}\Big)^{\delta(p)}_+=m_k(\l)+\eta_k(\l) n_k(\l),\quad\quad \lambda\in \mathbb{R},
\end{eqnarray}
such that:
\begin{itemize}
\item[(a)]  The functions $\widehat{m_k}$ and $\widehat{n_k}$ are even and supported
on $[-2^k/R,2^k/R]$;

\item[(b)]  The functions  $\eta_k$ are continuous and
$\sum_{k=1}^\infty|\eta_k(\l)|^2\leq C$  uniformly in $\l>0$ and in $R>0$.
In addition, we have that for $R>1$ and $\l>0$,
\begin{eqnarray}\label{e2.10}
\sum_{k=1}^\infty|\eta_k(\l)|^2(1+\l^2)^{\gamma}\leq CR^{2\gamma},\quad \ \ {\rm for\,all\,} \gamma>0
\end{eqnarray}
with $C$ independent of $\l$ and $R$;

\item[(c)] For arbitrary large   $N\in \NN$ there exists a constant $C=C(N)$ such that
$$
|n_k(\l)|\leq C2^{-\delta(p) k}\Big(1+2^k\Big|1-{|\l|\over R}\Big|\Big)^{-N}.
$$
\end{itemize}
\end{lemma}

\begin{proof}  We follow  \cite[Lemma 2.1]{T2} to obtain a decomposition
$S_R^{\delta(p)}(\l^2)=m_k(\l)+\eta_k(\l) n_k(\l)$ such that properties (a), (b) and (c) of Lemma~\ref{le2.4}
hold, except that inequality~\eqref{e2.10} in (b) remains   to be verified. Indeed, from the construction of $\eta_k$ in  \cite[Lemma 2.1]{T2},
it follows that  for $|1-|\l|/R|>2^{-k},$
$$
\eta_k(\l)\leq C_N\left(2^k\bigg|1-{|\l|\over R}\bigg|\right)^{-N}
$$
for each $N\in \NN$, and for $|1-|\l|/R|\leq 2^{-k}$,
$$
\eta_k(\l)\leq C\left(2^{-k}+2^k\bigg|1-{|\l|\over R}\bigg|\right)^{\eps}
$$
for some $\eps>0$. Then we write
\begin{eqnarray*}
\sum_{k=1}^\infty|\eta_k(\l)|^2(1+\l^2)^{\gamma}&\leq & C\sum_{k:\ 2^{-k} < |1-{|\l|\over R} |}
\left(2^k\big|1-{|\l|\over R}\big|\right)^{-2N}(1+\l^2)^{\gamma}\\
&&\quad+\,\,C\sum_{k:\ 2^{-k} \geq  |1-{|\l|\over R} | }  \left(2^{-k}+2^k\big|1-{|\l|\over R}\big|\right)^{2\eps}(1+\l^2)^{\gamma}\\
&=:&\textup{(I)}+\textup{(II)}.
\end{eqnarray*}
 Since $R>1$, we have
 \begin{eqnarray*}
\textup{(I)}&\leq&
\left\{
\begin{array}{ll}
C\sum_{ k:\ 2^{-k} < |1-{|\l|\over R} |  }
 \left(2^k\big|1-{|\l|\over R}\big|\right)^{-2N}   R^{2\gamma},& \ \ {\rm if} \ {|\lambda|\over R}<2;\\[8pt]
C \sum_{ k:\ 2^{-k} < |1-{|\l|\over R} |  }\left(2^k\big|1-{|\l|\over R}\big|\right)^{-2(N-\gamma)}
\left({|\l|\over R} \right)^{-2\gamma} 2^{-2\gamma k} \l^{2\gamma} ,& \ \ {\rm if} \ {|\lambda|\over R}\geq 2;
\end{array}
\right. \\
 &\leq & CR^{2\gamma},
\end{eqnarray*}
as long as $N\in \mathbb{N}$ is chosen so that~$N>\gamma$. For the term $\textup{(II)}$, we  note that $|1-|\l|/R|\leq 2^{-k}\leq 1$, and so  $|\l|\leq 2R$. Then we have
 \begin{eqnarray*}
\textup{(II)}&\leq& C\sum_{|1-|\l|/R|\leq 2^{-k}}\left(2^{-k}+2^k\big|1-{|\l|\over R}\big|\right)^{2\eps}R^{2\gamma}
 \leq  CR^{2\gamma}.
\end{eqnarray*}
This proves \eqref{e2.10}, and completes the proof of Lemma~\ref{le2.4}.
 \end{proof}

\medskip

\section{Proof of Theorem~\ref{th1.1}}\label{BRS}
\setcounter{equation}{0}

For clarity, we prove Theorem~\ref{th1.1} for the Hermite operator $H=-\Delta+|x|^2$. Actually the conclusion of Theorem~\ref{th1.1} also holds for $ H_V= - \Delta + V$ where $V$ satisfies \eqref{e2.1}. See Section 4 for details. For the proof, as stated in the introduction, we first use Calder\'on-Zygmund techniques to decompose the function $f$ as $f=g+\sum_j b_j$. For the ``good" part $g$ and for those $b_j$ that have small support, the argument is similar to that in \cite{T2} or that in \cite{COSY}. The main difference happens when the support of $b_j$ is large; here we apply Lemma~\ref{le2.2}. See estimate~\eqref{e3.10} and its proof  below for details.

\begin{proof}[Proof of Theorem~\ref{th1.1}]
First we consider the case  that $R\leq 4$.
Fix $n\geq 2$ and $p$ with $1\leq p\leq 2n/(n+2)$.
We show  that $S_R^{\delta(p)} (H)$ is of weak type $(p,p)$
uniformly in $R\leq 4$. In this case, we apply Lemma~\ref{le2.2} to obtain that for $\gamma=n(1/p-1/2)$,
\begin{eqnarray*}
\|S_R^{\delta(p)} (H)f\|_{L^{p,\infty}}
 &\leq&  \|(I+H)^{-\gamma/2}\|_{L^2\to L^{p,\infty}} \, \|S_R^{\delta(p)} (H)(1+H)^{\gamma/2} f\|_{2}\\
 &\leq& C\|S_R^{\delta(p)} (H)(1+H)^{\gamma/2} f\|_{2}.
\end{eqnarray*}
Since $R\leq 4$, we have $\supp\, S_R^{\delta(p)}(\lambda^2)\subset [-16, 16] $. So it follows from the Hermite expansion~\eqref{e1.4} and equality~\eqref{e1.5} that
 $$
 S_R^{\delta(p)} (H)(1+H)^{\gamma/2} f=\sum\limits_{k=0}^7 S_R^{\delta(p)} (2k+n)(1+2k+n)^{\gamma/2}  P_kf.
 $$
 We apply the above equality and the restriction estimate~\eqref{e1.7} to obtain
\begin{eqnarray*}
\|S_R^{\delta(p)} (H)f\|_{L^{p,\infty}}
&\leq& C\sum_{k=0}^7 S_R^{\delta(p)} (2k+n)(1+2k+n)^{\gamma/2} k^{\delta(p)/2-1/4}
\|f \|_{p}
\leq  C\|f\|_{p},
\end{eqnarray*}
as required.

Next we consider the remaining case $R>4$.
  Fix $f\in L^p$ and $\alpha> 0$,
and apply  the Calder\'on-Zygmund decomposition at height $\alpha$
to $|f|^p.$ There  exist  constants $C$ and  $K$ so that
\begin{enumerate}
\item[(i)] $
f=g+b=g+\sum_{j}b_j;
$

\vskip 1pt
\item[(ii)]   $\|g\|_p\leq C\|f\|_p,$
 $\|g\|_{\infty}\leq C\alpha$;

\vskip 1pt
\item[(iii)] $b_j$  is supported in  $B_j$ and  $\# \{j: x\in 4B_j\} \leq K$ for all $x\in \RN$;

\vskip 1pt
\item[(iv)]
$\int_{\RN}\ |b_j|^pdx\leq C\alpha^p |B_j|, $  and
$\sum_{j} |B_j| \leq C{\alpha}^{-p} \|f\|_p^p.$
\end{enumerate}
Note that by (ii), we have $\alpha^{p-2}\|g\|_2^2\leq C\|f\|_p^p.$

\smallskip

 Let $r_{B_j}$ be the radius of $B_j$ and let
 $$
 J_k:=\big\{j: \, 2^k/R \le r_{B_j}<2^{k+1}/R\big\}, \quad\text{for $k\in \mathbb{Z}$}.
 $$
 Write
\begin{eqnarray*}
f  = g+\sum_{j}b_j&=&
 g+\sum_{k\leq 0}\sum_{j\in J_k}b_j+\sum_{k>0}\sum_{j\in J_k}b_j\\
 &=:& g+h_1+h_2.
\end{eqnarray*}
Then it is enough to show that there exists a constant $C>0$ independent
of $R$ and $\alpha$ such that
\begin{equation}\label{e3.1}
\big |\{x:S_R^{\delta(p)}(H)(g)(x)>\alpha\}\big|
  \leq
C\alpha^{-p}\|f\|_p^p
\end{equation}
and  such that for $i=1,2,$
\begin{equation}\label{e3.2}
\left|\big\{x:S_R^{\delta(p)}(H)\big(h_i\big)(x)>\alpha\big\}\right|
  \leq
C\alpha^{-p}\|f\|_p^p.
\end{equation}

\medskip

 Note that $\sup_{\l,R>0}\Big(1-{\l^2/R^2}\Big)^{\delta(p)}_+ =1$  and that $\alpha^{p-2}\|g\|_2^2\leq C\|f\|_p^p.$
Hence by  the spectral
 theorem
\begin{eqnarray}\label{e3.3}
\Big |\{x:S_R^{\delta(p)}(H)(g)(x)>\alpha\}\big|&\leq&
\alpha^{-2}\|S_R^{\delta(p)}(H)(g)\|_2^2\leq \alpha^{-2}\|g\|_2^2
 \leq
C\alpha^{-p}\|f\|_p^p,
\end{eqnarray}
which proves~\eqref{e3.1}.

Next we  prove (\ref{e3.2})   for     $i=1$.
 By  the decomposition~\eqref{e2.9},
\begin{eqnarray*}
 \sum_{k\leq 0}\sum_{j\in
J_k}S_R^{\delta_q(p)}(H)b_j &=&
\sum_{k\leq
0}\eta_k({\SL})\Big(\sum_{j\in
J_k}n_k({\SL})b_j \Big)  \\
  & +&   S_R^{\delta(p)}(H)\Big(\sum_{k\leq 0}\sum_{j\in
J_k}n_k({\SL})b_j\Big).  \nonumber
\end{eqnarray*}
Applying   the spectral theorem and Lemma~\ref{le2.5}  with
$Q_k(\lambda)=\eta_k(\lambda)$ yields
\begin{eqnarray}\label{e3.4}
 \Big\|\sum_{k\leq 0}\sum_{j\in
J_k}S_R^{\delta(p)}(H)b_j\Big\|^2_2 &\leq& C\sum_{k\leq
0}\Big\|\sum_{j\in J_k}n_k({\SL })b_j\Big\|^2_2+C\Big\|\sum_{k\leq
0}\sum_{j\in J_k}n_k({\SL})b_j\Big\|^2_2.
\end{eqnarray}
Next,  since $\supp \widehat{n_k} \subseteq [-2^k/R,2^k/R]$, by Lemma~\ref{le2.1}, we have
$$
\supp K_{n_k({\sqrt{H} })} \subseteq \big\{(x,y)\in \RN\times
\RN:  |x-y|\leq 2^k/R\big\}.
$$
Hence if $j\in J_k$, then
$\supp n_k(\sqrt{H})b_j\subseteq  4B_j$. Thus by (iii) there exists
a constant $C>0$ such that
\begin{eqnarray} \hspace{1cm}
\sum_{k\leq
0}\Big\|\sum_{j\in J_k}n_k({\SL })b_j\Big\|^2_2+\Big\|\sum_{k\leq
0}\sum_{j\in J_k}n_k({\SL})b_j\Big\|^2_2  \le  C\sum_{k\leq 0}\sum_{j\in
J_k}\Big\|n_k({\SL})b_j\Big\|^2_2.\label{e3.5c}
\end{eqnarray}
Next, noting that $R>4$ and $k\leq 0$, by part (c) of Lemma~\ref{le2.3} and the restriction type estimate~\eqref{e1.7}
\begin{align}\label{e3.5}
\Big\|n_k({\SL})b_j\Big\|^2_2
=\Big\lag n_k^2({\SL})b_j,b_j\Big\rag
&=\Big\lag \sum_{\ell\in \mathbb{N}} n_k^2(\sqrt{2\ell+n})P_\ell b_j,b_j\Big\rag\nonumber\\
&\leq C \sum_{\ell\in \mathbb{N}} \left(1+\frac{2^k\sqrt{2\ell+n}}{R}\right)^{-2N}\|P_\ell b_j\|_2^2\nonumber\\
&\leq C \sum_{\ell\in \mathbb{N}} \left(1+\frac{2^k\sqrt{2\ell+n}}{R}\right)^{-2N}\ell^{\delta(p)-1/2}\| b_j\|_p^2.
\end{align}
Since
\begin{eqnarray*}
 \sum_{\ell\in \mathbb{N}} \left(1+\frac{2^k\sqrt{2\ell+n}}{R}\right)^{-2N}\ell^{\delta(p)-1/2}
&\leq& \sum_{\ell\in \mathbb{N}} \left(1+\frac{2^k\sqrt{\ell}}{R}\right)^{-2N}\ell^{\delta(p)-1/2}\\
&\leq& \sum_{\ell\geq (R/2^k)^2} \left(\frac{2^k\sqrt{\ell}}{R}\right)^{-2N}\ell^{\delta(p)-1/2}+\sum_{0<\ell< (R/2^k)^2} \ell^{\delta(p)-1/2}\\
&\leq& C\left(\frac{2^k}{R}\right)^{2n(1/2-1/p)},
\end{eqnarray*}
we have
$$
\Big\|n_k({\SL})b_j\Big\|_2\leq C\left(\frac{2^k}{R}\right)^{n(1/2-1/p)}\|b_j\|_p\leq C|B_j|^{1/2-1/p}\|b_j\|_p\leq C \alpha |B_j|^{1/2}.
$$
Hence by \eqref{e3.4}, \eqref{e3.5c} and (iv),
\begin{eqnarray}\label {e3.6}\hspace{1cm}
 \Big|\Big\{x:\big|S_R^{\delta(p)}(H)\Big(\sum_{k\leq
0}\sum_{j\in J_k}b_j\Big)\big|>\alpha\Big\}\Big|
&\leq&C\alpha^{-2}\Big\|S_R^{\delta(p)}(H)\Big(\sum_{k\leq
0}\sum_{j\in
J_k}b_j\Big)\Big\|^2_2
 \leq C \alpha^{-p}\|f\|_p^p,
\end{eqnarray}
which proves~\eqref{e3.2} for $i=1$.
\medskip

Now, we prove (\ref{e3.2})    for    $i=2$.
Let $\Omega^*:=\bigcup_{j\in \mathbb{N}}4B_j$. \noindent
From (iii) and (iv), it follows that
 $$
 |\Omega^*| \leq
C\sum_{j} |B_j|\leq C \alpha^{-p}\|f\|_p^p.
$$
 Hence it is enough to show that
\begin{eqnarray}\label{e3.7}
\bigg|\Big\{x\in \RN \backslash\Omega^*:\Big|S_R^{\delta(p)}(H)\Big(\sum_{k>
0}\sum_{j\in J_k}b_j\Big)\Big|>\alpha\Big\}\bigg| \leq C \alpha^{-p}\|f\|_p^p.
\end{eqnarray}

\medskip

Using the decomposition from  Lemma~\ref{le2.4} we write
\begin{eqnarray}\label{e3.8} \hspace{1.5cm}
S_R^{\delta(p)}(H)\Big(\sum_{k> 0}\sum_{j\in J_k}b_j\Big)=
\sum_{k> 0}\sum_{j\in J_k}m_k(\SL)b_j +\sum_{k>
0}\eta_k(\SL)n_k(\SL)\Big(\sum_{j\in J_k}b_j\Big).
\end{eqnarray}
Recall that $\widehat{m_k}$ is even
and supported in $[-2^k/R,2^k/R]$. By  Lemma~\ref{le2.1},
$$
\supp K_{m_k(\sqrt{H})} \subset \Big\{(x,y)\in \RN\times
\RN: |x-y|\leq {2^k/R}\Big\}.
$$
This implies that if  $x\in \RN\backslash\Omega^*$, then $
m_k(\sqrt{H})b_j(x)=0$ for each $j\in J_k $ and $k>0$. So the first term on the right hand of equality~\eqref{e3.8} makes
no contribution to estimate~\eqref{e3.7}.  So the proof  of  \eqref{e3.7} reduces to showing that
\begin{eqnarray}\label{e3.9}
\Big |\Big\{x:\Big|\sum_{k>
0,\,2^k\leq R^2}\eta_k(\SL)n_k(\SL)\Big(\sum_{j\in J_k}b_j\Big)\Big|>\alpha\Big\}\Big| \leq C \alpha^{-p}\|f\|_p^p
\end{eqnarray}
and
\begin{eqnarray}\label{e3.10}
\Big |\Big\{x:\Big|\sum_{k>
0,\,2^k> R^2}\eta_k(\SL)n_k(\SL)\Big(\sum_{j\in J_k}b_j\Big)\Big|>\alpha\Big\}\Big| \leq C \alpha^{-p}\|f\|_p^p.
\end{eqnarray}

We claim that
\begin{eqnarray}\label{e3.11}
\Big\|n_k(\SL)b_j\Big\|_2 &\leq& C
\alpha |B_j|^{1/2}\max\{2^{k/2}R^{-1},1\}, \quad\text{for $j\in J_k$, $k>0$}.
\end{eqnarray}
We note that for operators satisfying~\eqref{sp} on compact manifolds, or for operators satisfying~\eqref{rp}, there is no need for the extra factor $\max\{2^{k/2}R^{-1},1\}$ in the above estimate. For the Hermite operator, this extra factor may be unavoidable.

Before we prove estimate~\eqref{e3.11}, let us see how it implies \eqref{e3.9} and \eqref{e3.10}.
We handle \eqref{e3.9} first.

\smallskip

\noindent
{\bf Estimate for \eqref{e3.9}}. This estimate follows from a similar argument to that in \cite{T2} or that in \cite{COSY}.  By Lemma~\ref{le2.5} and part (b) of Lemma~\ref{le2.4},
\begin{eqnarray}\label{e3.12}
\Big |\Big\{x:\Big|\sum_{k>
0,\,2^k\leq R^2}\eta_k(\SL)n_k(\SL)\Big(\sum_{j\in J_k}b_j\Big)\Big|>\alpha\Big\}\Big|
&\leq& \alpha^{-2} \Big\|\sum_{k>
0,\,2^k\leq R^2}\eta_k(\SL)n_k(\SL)\Big(\sum_{j\in J_k}b_j\Big)\Big\|_2^2\nonumber\\
&\leq&C\alpha^{-2} \sum_{k>
0,\,2^k\leq R^2} \Big\|n_k(\SL)\Big(\sum_{j\in J_k}b_j\Big)\Big\|_2^2.
\end{eqnarray}
Next,
 since $\widehat{n_k}$ is  even and supported on $[-2^k/R,2^k/R]$,
by    Lemma~\ref{le2.1}
  $$\supp K_{n_k(\sqrt{H})} \subseteq \big\{(x,y)\in \RN\times
\RN: |x-y|\leq 2^k/R\big\}.
$$
Hence
 $\supp n_k(\sqrt{H})b_j\subseteq 4B_j$ for each $j\in J_k$. By (iii) and \eqref{e3.11}
 there exists a constant $C>0$ such that
\begin{eqnarray*}
\sum_{k>0,\,2^k\leq R^2} \big\|n_k(\SL)\Big( \sum_{j\in J_k}b_j\Big)\big\|_2^2
&\leq& C\sum_{k>0,\,2^k\leq R^2}  \sum_{j\in J_k}\big\|n_k(\sqrt{H})b_j\big\|_2^2 \nonumber\\
 &\leq& C\sum_{k>0,\,2^k\leq R^2}  \sum_{j\in J_k} \alpha^2 |B_j|\max\{2^{k}R^{-2},1\}\\
 &=& C\sum_{k>0,\,2^k\leq R^2}  \sum_{j\in J_k} \alpha^2 |B_j|\\
&\leq& C\alpha^2 \sum_{j} |B_j|
 \leq C  \alpha^{2-p}\|f\|_p^p,
\end{eqnarray*}
where in the last inequality we have used (iv). This, in combination with \eqref{e3.12}, implies \eqref{e3.9}.

\smallskip

\noindent
{\bf Estimate for \eqref{e3.10}}.  As explained in the introduction, because of the extra factor $\max\{2^{k/2}R^{-1},1\}$ in estimate~\eqref{e3.11}, the proof of this estimate relies on the \emph{a priori} estimate~\eqref{e1.8}. Recall that $\gamma=n(1/p-1/2)$ where $1\leq p\leq 2n/(n+2).$
We apply Lemma~\ref{le2.2} (which states estimate~\eqref{e1.8}), \eqref{le2.5} and inequality~\eqref{e2.10}   to obtain
\begin{eqnarray}\label{e3.13}
\lefteqn{ \bigg|\Big\{x:\Big|\sum_{k>
0,\,2^k> R^2}\eta_k(\SL)n_k(\SL)\Big(\sum_{j\in J_k}b_j\Big)\Big|>\alpha\Big\}\bigg|}\nonumber\\
&\leq& C \alpha^{-p} \|(1+H)^{-\gamma/2}\|_{L^2\to L^{p,\infty}} \Big\|\sum_{k>
0,\,2^k> R^2}\eta_k(\SL)(1+H)^{\gamma/2}n_k(\SL)\Big(\sum_{j\in J_k}b_j\Big)\Big\|_2^p\nonumber\\
&\leq&C\alpha^{-p} \left(\sum_{k>
0,\,2^k> R^2} R^{2\gamma}\Big\|n_k(\SL)\Big(\sum_{j\in J_k}b_j\Big)\Big\|_2^2\right)^{p/2}.
\end{eqnarray}
As noted above, $\supp n_k(\sqrt{L})b_j\subseteq 4B_j$ for each $j\in J_k$. By (iii) of the Calder\'on-Zygmund decomposition of $|f|$ and \eqref{e3.11}
 there exists a constant $C>0$ such that
\begin{eqnarray*}
\sum_{k>
0,\,2^k> R^2} \big\|n_k(\SL)\Big( \sum_{j\in J_k}b_j\Big)\big\|_2^2
&\leq& C\sum_{k>
0,\,2^k> R^2}  \sum_{j\in
J_k}\big\|n_k(\SL)b_j\big\|_2^2\\
&\leq&C\sum_{k>
0,\,2^k> R^2}  \sum_{j\in
J_k} \alpha^2 |B_j|\max\{2^{k}R^{-2},1\}\\
&=&C\sum_{k>
0,\,2^k> R^2}  \sum_{j\in
J_k} \alpha^2 |B_j|2^{k}R^{-2}.
\end{eqnarray*}
This estimate, in combination with  the fact that $p/2\leq 1$, shows that
\begin{eqnarray*}
{\rm RHS\ of \ } \eqref{e3.13}
&\leq&C\alpha^{-p} \left(\sum_{k>
0,\,2^k> R^2}\sum_{j\in
J_k}  R^{2\gamma}\alpha^2 |B_j|2^{k}R^{-2}\right)^{p/2}\\
&\leq& C \sum_{k>
0,\,2^k> R^2}\sum_{j\in
J_k}   R^{\gamma p} \Big(\frac{2^k}{R}\Big)^{np/2}2^{kp/2}R^{- p}.
\end{eqnarray*}
On the other hand, since $2^k>R^2$, $R>4$, $\gamma=n(1/p-1/2)$  and $n(1/p-1/2)-1/2\geq 0$, we have
\begin{eqnarray*}
 R^{\gamma p} \Bigg(\frac{2^k}{R}\Bigg)^{np/2}2^{kp/2}R^{- p}
 &=& R^{p(\gamma-{1\over  2})} \left(\frac{2^k}{R}\right)^{({n\over 2}+{1\over 2})p}
 \leq  \left(\frac{2^k}{R}\right)^{p (n({1\over p} -{1\over 2})- {1\over  2})+({n\over 2}+{1\over 2})p}
  \leq  \left(\frac{2^k}{R}\right)^{n}
  \leq  C |B_j|,
\end{eqnarray*}
which implies
\begin{eqnarray*}
{\rm LHS\ of \ } \eqref{e3.13}
&\leq&   C \sum_{j}    |B_j|
 \leq C\alpha^{-p}\|f\|_p^p
\end{eqnarray*}
as desired in \eqref{e3.10}.

\smallskip

It remains only to prove \eqref{e3.11}.  Note that in this case $k> 0$ and $R>4$.

By the Hermite expansion and the functional calculus of $H$, we can rewrite
$$
n_k({\SL})b_j=\sum_{\ell\in \mathbb{N}} n_k(\sqrt{2\ell+n})P_\ell b_j,
$$
where $P_\ell$ are the projections defined by~\eqref{e1.6}. Then by part (c) of Lemma~\ref{le2.4} and
the restriction type estimate~\eqref{e1.7}, we see that
\begin{eqnarray}\label{e3.14}
\Big\|n_k({\SL})b_j\Big\|^2_2
&=&\Big\lag \sum_{\ell\in \mathbb{N}} n_k^2(\sqrt{2\ell+n})P_\ell b_j,b_j\Big\rag\nonumber\\
&\leq& C \sum_{\ell\in \mathbb{N}} 2^{-2\delta(p)k}\Bigg(1+2^k\Big|{\sqrt{2\ell+n}\over R}-1\Big|\Bigg)^{-2N}
\|P_\ell b_j\|_2^2\nonumber\\
&\leq&C \sum_{\ell\in \mathbb{N}} 2^{-2\delta(p)k}\left(1+2^k\Big|{\sqrt{2\ell+n}\over R}-1\Big|\right)^{-2N}\ell^{\delta(p)-1/2}\| b_j\|_p^2\nonumber\\
&\leq& C\sum_{\ell\in \mathbb{N}} 2^{-2\delta(p)k}\left(1+2^k\Big|{\sqrt{\ell}\over R}-1\Big|\right)^{-2N}\ell^{\delta(p)-1/2}\| b_j\|_p^2 .
\end{eqnarray}
We split this sum into three parts:
\begin{eqnarray} \label{e3.15}
\sum_{\ell\in \mathbb{N}} &=&
    \sum_{R^2(1-2^{-k})^2-1< \ell< R^2(1+2^{-k})^2+1}
	 +
 \sum_{\ell\ge R^2(1+2^{-k})^2+1}
 +\sum_{0<\ell\le R^2(1-2^{-k})^2-1}\nonumber\\
 &=:& \textup{(I)}+\textup{(II)}+\textup{(III)}.
\end{eqnarray}
For $\textup{(I)}$, we can control each term in the summation by the same bound, namely $C2^{-2\delta(p)k}R^{2\delta(p)-1}$, because for $\ell$ in this range, the expression with exponent $-2N$ is almost 1.
 So the key point is to count how many terms there are in the summation. If $2^k\leq R^2$, then there are at most $R^22^{-k}$ terms (up to multiplication by an absolute constant) in the summation, and if $2^k>R^2$, then there are at most six terms in the summation. Thus we see that
\begin{eqnarray*}
\textup{(I)}&\leq&
\left\{
\begin{array}{rr}
C2^{-2\delta(p)k}R^{2\delta(p)-1},& \ \ {\rm if} \ 2^k> R^2;\\[8pt]
C R^2 2^{-k}2^{-2\delta(p)k}R^{2\delta(p)-1},& \ \ {\rm if} \ 2^k\leq R^2;
\end{array}
\right. \\
&=& \left\{
\begin{array}{rr}
C\left(\frac{2^k}{R}\right)^{2n(1/2-1/p)} \frac{2^k}{R^2},& \ \ {\rm if} \ 2^k> R^2;\\[8pt]
C\left(\frac{2^k}{R}\right)^{2n(1/2-1/p)},& \ \ {\rm if} \ 2^k\leq R^2;
\end{array}
\right. \\
 &=&  C\left(\frac{2^k}{R}\right)^{2n(1/2-1/p)} \max\Big\{1, \frac{2^k}{R^2}\Big\}.
\end{eqnarray*}
We briefly highlight why the extra factor $\max\Big\{1, 2^k/R^2\Big\}$ is present here.
When the manifold on which $f$ is defined is compact, since $R>4$, $2^k/R^2$ is less than $2^k/R$ which is the radius of the support of $b_j$ and so $2^k/R^2$ is less than the diameter of the manifold. So in this situation the factor $\max\Big\{1, 2^k/R^2\Big\}$ is controlled by an absolute constant. When the operator $L$ satisfies the restriction estimate~\eqref{rp}, we can use integration over the continuous spectrum, instead of summation over the eigenvalues, in the expression for $n_k({\sqrt L})b_j$. Then when $2^k/R^2$ is large, the interval of integration is correspondingly small. So in this situation, the extra factor $\max\Big\{1, 2^k/R^2\Big\}$ is canceled out, by a factor involving the length of the interval of integration. However, for our Hermite operator $H$, no matter how small the interval of integration or summation, there still may be an eigenvalue in it, so we do have the extra factor $\max\Big\{1, 2^k/R^2\Big\}$ in our estimate.

To estimate  the term  $\textup{(II)}$, we note that  the function
$$
x^{\delta(p)-1/2}\bigg(2^k\Big({\sqrt{x}\over R}-1\Big)\bigg)^{-2N}
$$
is decreasing for $x>R^2$ and $N$ sufficiently large, and thus
\begin{eqnarray*}
\textup{(II)}&\leq& \int_{R^2(1+2^{-k})^2}^\infty 2^{-2\delta(p)k}\bigg(2^k\Big({\sqrt{x}\over R}-1\Big)\bigg)^{-2N}x^{\delta(p)-1/2}dx\\
&\leq& C 2^{-2\delta(p)k} 2^{-2Nk} R^{2\delta(p)+1}\int_{1+2^{-k}}^\infty  \big(t-1\big)^{-2N}t^{2\delta(p)}dt\\
&\leq& C \left(\frac{2^k}{R}\right)^{2n(1/2-1/p)}.
\end{eqnarray*}
By symmetry, a similar  argument to that in $\textup{(II)}$ shows that $\textup{(III)}\leq C \left(2^k/R\right)^{2n(1/2-1/p)}.$

Collecting the estimates of the terms $\textup{(I)}$, $\textup{(II)}$ and $\textup{(III)}$,  together  with   \eqref{e3.14},
(iv) of Calder\'on-Zygmund decomposition of function $f$ and the fact $j\in J_k$, we arrive at the conclusion
that
$$
\Big\|n_k({\SL})b_j\Big\|_2\leq C\left(\frac{2^k}{R}\right)^{n(1/2-1/p)} \max\Big\{1,  {2^{k/2}}{R}^{-1}\Big\}
\|b_j\|_p\leq C \alpha |B_j|^{1/2}\max\Big\{1,  {2^{k/2}}{R}^{-1}\Big\}.
$$
This proves  \eqref{e3.11}, and completes the proof of Theorem~\ref{th1.1}.
\end{proof}

\medskip

\section{Extensions}
\setcounter{equation}{0}
In the previous section, we proved Theorem~\ref{th1.1}, where the potential is $V=|x|^2.$
However, the precise form of this potential does not play a fundamental role in the estimates.
Here we consider instead the operators
  $ H_V= - \Delta + V$
with a  positive potential $V$ which satisfies the following conditions:
\begin{equation}\label{e4.1}
V \sim |x|^2, \quad |\nabla V| \sim |x|, \quad |\partial_x^2 V| \le 1.
\end{equation}
Under these assumptions the operator $H_V$
is a nonnegative self-adjoint operator acting on the
space $L^2$.   Such an operator admits a spectral
resolution
\begin{eqnarray*}
H_V=\int_0^{\infty} \lambda dE_{H_V}(\lambda).
\end{eqnarray*}
Now, the Bochner-Riesz means of order $\delta\geq 0$   can be defined by
 \begin{equation}\label{e4.2}
 S^{\delta}_R(H_V) f :=  \int_0^{R^2} \left(1-\frac{\lambda}{R^2}\right)^{\delta}dE_{H_V}(\lambda) f, \ \ \ \ f\in L^2.
   \end{equation}
Then, just as for the Hermite operator $H$, the Bochner-Riesz means $S_R^\delta(H_V)$ are of weak-type $(p,p)$ uniformly
in $R>0$, as we now show.

 \begin{theorem}\label{th4.1} \, Suppose the potential $V$ satisfies~\eqref{e4.1}. For $n\geq 2$ and  $1\leq p\leq 2n/(n+2),$ there is a constant $C$ independent of $R$ for which
   \begin{eqnarray*}
   \|S_R^{\delta(p)}(H_V) f\|_{L^{p, \infty}} \leq C\|f\|_{p}.
   \end{eqnarray*}
\end{theorem}

\begin{proof}
It follows from Theorem 4 in  \cite{KoT} that  for all $\lambda \ge 0$  and
 all $1\leq p\leq 2n/(n+2)$
\begin{eqnarray}\label{e4.4}
\|E_{H_V}[\lambda^2,\, \lambda^2+1)\|_{p\to 2} \leq C(1+\lambda)^{n({1\over p}-{1\over 2})-1}.
\end{eqnarray}
With Lemma~\ref{le2.2} and the spectral projection  estimate \eqref{e4.4},  the argument   in the proof of Theorem~\ref{th1.1} also establishes Theorem~\ref{th4.1}.
\end{proof}

\bigskip

 \noindent
 {\bf Acknowledgements:}
P. Chen and L.A. Ward were supported by the Australian Research Council, Grant No.~ARC-DP160100153. P. Chen was also supported by NNSF of China, Grant No.~11501583, Guangdong Natural Science Foundation, Grant No.~2016A030313351 and the Fundamental Research Funds
for the Central Universities 161gpy45. J. Li was supported by the Australian Research Council, Grant No.~ARC-DP170100160, and by Macquarie University Research Seeding Grant.
 L.X. Yan was supported by the NNSF
of China, Grant No.~11471338, ~11521101 and ~11871480, and Guangdong Special Support Program. The authors would like to thank L. Song for helpful discussions.

\end{document}